\documentclass[12pt]{amsart}
\usepackage{amsmath,amssymb,amsfonts,amsthm,amsopn}
\usepackage{graphicx,tikz}
\usepackage[all]{xy}
\usepackage{amsmath}
\usepackage{multirow, longtable, makecell, caption, array,enumitem}
\usepackage{amssymb}
\usepackage{mathtools}
\mathtoolsset{showonlyrefs}

\usepackage{bookmark}
\usepackage{hyperref,color}
\calclayout

\usepackage{geometry}
\newcommand{\adm}[1]{{\left\vert\kern-0.25ex\left\vert\kern-0.25ex\left\vert #1 
		\right\vert\kern-0.25ex\right\vert\kern-0.25ex\right\vert}}

\newtheorem{theorem}{Theorem}[section]

\newtheorem{corollary}[theorem]{Corollary}
\newtheorem{proposition}[theorem]{Proposition}
\newtheorem{definition}[theorem]{Definition}

\theoremstyle{definition}

\newtheorem{example}[theorem]{Example}
\newtheorem{remark}[theorem]{Remark}

\theoremstyle{definition}

\usepackage{cellspace} %
\def\bR{\mathbb{R}}
\def\bC{\mathbb{C}}
\def\bN{\mathbb{N}}

\def\smo{\setminus\{0\}}

\def\cF{\mathcal{F}}
\def\cS{\mathcal{S}}

\def\rd{\bR^d}

\def\rdd{\bR^{2d}}

\def\lan{\langle}
\def\ran{\rangle}

\def\supp{\text{supp}}
\def\wt{\widetilde}

\def\S0{S^0_{0,0}}

\def\Bd'{B_{\delta'}}

\def\cBd'{\bar{B}_{\delta'}}

\def\vp{\varphi}
\def\veps{\varepsilon}
\def\wh{\widehat}
\def\ird{\int_{\rd}}

\def\smo{\setminus\{0\}}

\def\lax{\lan x \ran}
\def\Fmu{F_\mu}

\def\XXint#1#2#3{{\setbox0=\hbox{$#1{#2#3}{\int}$ }
		\vcenter{\hbox{$#2#3$ }}\kern-.6\wd0}}

\newcommand*\dd[1]{\mathop{}\!\mathrm{d}#1}

\def\diam{\text{diam}}
\begin{document}
	
	\title[Phase space analysis of higher-order dispersive equations]{Phase space analysis of higher-order dispersive equations with point interactions}
	
	\author[S. Mazzucchi]{Sonia Mazzucchi}
	\address{Dipartimento di Matematica, Università di Trento, Via Sommarive 14, 38123 Povo, Italy}
	\email{sonia.mazzucchi@unitn.it}
	
	\author[F. Nicola]{Fabio Nicola}
	\address{Dipartimento di Scienze Matematiche ``G. L. Lagrange'', Politecnico di Torino, corso Duca degli Abruzzi 24, 10129 Torino, Italy}
	\email{fabio.nicola@polito.it}
	
	\author[S.\ I. Trapasso]{S. Ivan Trapasso}
	\address{Dipartimento di Scienze Matematiche ``G. L. Lagrange'', Politecnico di Torino, corso Duca degli Abruzzi 24, 10129 Torino, Italy}
	\email{salvatore.trapasso@polito.it}
	
	
	\begin{abstract}
		We investigate nonlinear, higher-order dispersive equations with measure (or even less regular) potentials and initial data with low regularity. Our approach is of distributional nature and relies on the phase space analysis (via Gabor wave packets) of the corresponding fundamental solution --- in fact, locating the modulation/amalgam space regularity of such generalized Fresnel-type oscillatory functions is a problem of independent interest in harmonic analysis. 
	\end{abstract}
	\maketitle

\section{Introduction}
This note concerns the wave packet analysis of higher-order dispersive equations with measures (or even more singular distributions) as potentials. In the first place, let us consider time-independent potentials and linear problems of the form
\begin{equation}\label{eq-intro-pbm-lin}
		\begin{cases}
			\frac{1}{2\pi i} \partial_t u = \mu(D)u + V(x)u \\
			u(0,x)=f(x),
		\end{cases} \qquad (t,x) \in \bR_+ \times \rd,
	\end{equation} where:
 \begin{itemize}
     \item $\mu(D)= \cF^{-1} \mu \cF$ is the Fourier multiplier with symbol $\mu \colon \rd\smo \to \bR$, which is a smooth, positively homogeneous function of degree $m\ge 2$ (in fact sufficiently large, depending on the regularity of $V$; see below), satisfying a non-degeneracy condition on the unit sphere (namely, $\det |\nabla^2 \mu(x)| \ne 0$ if $|x|=1$) to ensure the dispersive nature of the propagator $U_\mu(t) = e^{2\pi i t \mu (D)}$. 
     \item The potential $V$ is a temperate distribution that locally belongs to $\cF L^\infty(\rd)$ and is mildly  localized --- to be concrete, consider the case where $V$ is a finite, complex Borel measure on $\rd$. 
 \end{itemize}
A well-established approach to study (single or multi-source) point interaction models for the Schr\"odinger equation (viz., $\mu(x)=|x|^2$ and $u,f \in L^2(\rd)$) largely relies on tools from the spectral theory of operators on Hilbert spaces, including for instance self-adjoint extensions and form methods --- see for instance \cite{adami,albeverio_solv,brasche,christ,dellant1,dellant2,gesztesy} and the references therein. For higher-order equations as above, it seems more convenient to look at \eqref{eq-intro-pbm-lin} from a different angle, through the lens of distribution theory. Given that $V \in (\cF L^\infty(\rd))_{\rm loc}$, it is then natural to investigate the existence of solutions locally in $\cF L^1(\rd)$, so that the perturbation term $Vu$ is a well-defined distribution in $(\cF L^\infty(\rd))_{\text{loc}}$.

Standard techniques from the theory of dispersive PDE, combining Duhamel's integral representation of \eqref{eq-intro-pbm-lin} with a suitable fixed-point theorem (cf.\ \cite{tao book}), motivate a preliminary analysis of the unperturbed equation 
\begin{equation}\label{eq-intro-pbm-free}
					\begin{cases}
			\frac{1}{2\pi i} \partial_t u = \mu(D) u \\
			u(0,x)=f(x),
		\end{cases} \qquad (t,x) \in \bR_+ \times \rd.
	\end{equation} 
 
It is not difficult to see that the solution $u(t)$ is given by $u(t)=E(t,\cdot) \ast f$ for $t>0$, where $E \in C^\infty ((0,+\infty);\cS'(\rd))$ is the fundamental solution of \eqref{eq-intro-pbm-free}, hence satisfying $E(t,\cdot) \to \delta$ in $\cS'(\bR^d)$ as $t \to 0^+$. More in detail, we have \begin{equation} E(t,x)=t^{-1/m} E_\mu(t^{-1/m}x), \end{equation} where $E_\mu= \cF^{-1} F_\mu$ and $F_\mu(x) = e^{2\pi i \mu(x)}$. By comparison with the standard Schr\"odinger setting, it is natural to refer to $F_\mu$ as the \textit{generalized Fresnel function} associated with $\mu$. 

In light of the previous discussion, it would be desirable for the fundamental solution $E(t,\cdot)$ to belong to $\cF L^1(\rd)$ for any given $t>0$ (that is, $F_\mu\in \cF L^1(\bR^d)$), but this is not the case, as already evident from the Schr\"odinger equation. Nonetheless, our main result shows that $E_\mu$ is locally in $\cF L^1(\rd)$, uniformly with respect to translations in $\rd$. More precisely:
\begin{theorem}\label{thm-main}
		Let $\mu \colon \rd\smo \to \bR$ be a smooth, positively homogeneous function of degree $m \ge 2$, satisfying
		\begin{equation}
			\det |\nabla^2 \mu (x)| \ne 0 \quad \text{ if } \quad |x|=1. 
		\end{equation}
Let $F_\mu(x)=e^{2\pi i \mu(x)}$ be the generalized Fresnel function associated with $\mu$ and set $E_\mu = \cF^{-1} F_\mu$. For every $g \in \cS(\rd)$ we have 
\begin{equation}\label{eq-intro-emu}
   \sup_{y \in \rd} \| g(\cdot - y) E_\mu \|_{\cF L^1} < \infty.
\end{equation}
\end{theorem}
  As a consequence (cf.\ Corollary \ref{cor-disp-est}), there exists a constant $C>0$ (depending on $\mu$ and $g$) such that \begin{equation}\label{eq-intro-emu bis}
   \| E(t,\cdot) \|_{W^{1,\infty}} \coloneqq \sup_{y \in \rd} \| g(\cdot - y) E(t,\cdot) \|_{\cF L^1} \leq C\max\{t^{-2d/m},t^{-d/m}\},\qquad t>0.
\end{equation}
In the case of the Schr\"odinger equation (where $\mu(x)=|x|^2$), it is enough to choose $g$ as a Gaussian function to realize by explicit computations that the bound \eqref{eq-intro-emu bis} is actually optimal both as $t\to0^+$ and $t\to+\infty$ --- see, e.g., \cite{CN_str}. 

The condition appearing in \eqref{eq-intro-emu} actually shows that $E_\mu$ belongs to the so-called \textit{Wiener amalgam space} $W^{1,\infty}(\rd)$ --- or, equivalently, that the generalized Fresnel function $F_\mu$ belongs to the \textit{modulation space} $M^{1,\infty}(\rd)$. This terminology comes from harmonic analysis in phase space \cite{cordoba,fefferman,fei_mod83,folland, gro_book,tataru_cont}, where wave packet decompositions of functions and operators are performed and the regularity of distributions is measured in terms of mixed Lebesgue norms of the corresponding coefficients. To be more precise, recall that \textit{Gabor wave packets} are obtained by means of phase space translations of a generating atom $g \in \cS(\rd)$, namely
\begin{equation}
    \pi(x,\xi)g(y) \coloneqq e^{2\pi i \xi \cdot y} g(y-x), \qquad (x,\xi) \in \rdd. 
\end{equation}
The Gabor (or Bargmann, or short-time Fourier) transform of $\cS'(\rd)$ is then the phase space representation defined by
\begin{equation}
    V_g f (x,\xi) \coloneqq \lan f, \pi(x,\xi) g\ran = \ird e^{-2\pi i \xi \cdot y} f(y) \overline{g(y-x)} \dd{y}. 
\end{equation}
It is then clear that \eqref{eq-intro-emu} means that $V_g E_\mu(x,\xi) \in L^\infty_x(L^1_\xi)$, which is in fact the characterizing property of the Wiener amalgam space $W^{p,q}(\rd)$ with $p=1$ and $q= \infty$. For general $1 \le p,q \le \infty$, these are Banach spaces of distributions endowed with norms $\| f \|_{W^{p,q}} = \|V_g f(x,\xi) \|_{L^q_x(L^p_\xi)}$, which are independent of the choice of $g$ in practice (i.e., up to equivalence constants). Interchanging the phase-space integration order captures the modulation space regularity: $M^{p,q}(\rd)$ is the space of distributions $f \in \cS'(\rd)$ such that $\|f\|_{M^{p,q}} = \|V_g f(x,\xi)\|_{L^q_\xi(L^p_x)}$ is finite for some (hence any) non-trivial $g \in \cS(\rd)$ --- for instance, $F_\mu \in M^{1,\infty}(\rd)$, since $W^{p,q}(\rd) = \cF M^{p,q}(\rd)$. A review of the basic results of Gabor analysis is provided in Section \ref{sec-gabor}. 

The analysis of PDE with low-regular coefficients via suitable wave packets has a long tradition --- some fundamental contributions are to be found in \cite{smith,staff,tataru_ii}. More generally, advanced techniques from harmonic analysis have proven to be highly effective to the study of higher-order dispersive PDE \cite{kato,kenig,ruzh1,ruzh2}. Although the literature on nonlinear dispersive PDE is extensive and marked by impressive, sophisticated wellposedness results in individual cases, our goal here is rather to obtain general results for non-specific higher-order dispersive PDE with fairly irregular potentials.

We also stress that the analysis of the amalgam/modulation space regularity of oscillating functions such as $F_\mu$ is a largely investigated problem in harmonic analysis, in view of the relevance to the continuity of Schr\"odinger-type semigroups beyond Lebesgue spaces --- see for instance \cite{benyi_unimod, cunanan, guo, miyachi,NPT_mult,zhao}, and also Corollary \ref{cor-disp-est} below. In particular, it is well known that the unimodular function $e^{2\pi i |x|^m}$ belongs to $W^{1,\infty}(\rd)$ if $m \le 2$, and this condition is in fact optimal. Theorem \ref{thm-main} fills the gap on the phase space regularity in the case where $m\ge 2$, showing that $e^{2\pi i |x|^m} \in M^{1,\infty}(\rd)$ in that case. The Fresnel function $e^{2\pi i |x|^2}$ thus marks the edge between the two regimes.

After this brief detour, let us come back to the original problem with fresh eyes --- that is, studying local and global wellposedness of the higher-order dispersive problem in \eqref{eq-intro-pbm-lin} in the context of modulation and amalgam spaces. We highlight that the analysis of nonlinear dispersive PDE has indeed benefited from ideas and techniques of Gabor analysis in the last decades --- a comprehensive bibliography is inconceivable, we just mention \cite{bhimani_oko,chai,CN_str,CN_JFA,CNR_rough,fei_navier,takaoka,trapasso} and the monographs \cite{benyi_book,CR_book,gro_book,NT_book,wang_book}, bearing witness to the broad scope of this approach. 

As already anticipated, we plan to ignite a standard Duhamel-type iterative scheme starting from Theorem \ref{thm-main}. To this aim, we note from \eqref{eq-intro-emu bis} that the map $(0,1] \ni t \mapsto \|E(t,\cdot)\|_{W^{1,\infty}}$ is integrable if $m>2d$. A closer inspection of the details involved in the iteration argument actually shows that more general potentials in the class of the amalgam spaces can be taken into account, also in the presence of analytic nonlinearities. Our main result can thus be stated in full generality as follows.

\begin{theorem}\label{thm-pde}
	Let $\mu$ be as in Theorem \ref{thm-main} with $m>2d$. Given $V \in L^\infty([0,1];W^{\infty,1}(\rd))$ and an entire real-analytic\footnote{Recall that $G \colon \bC \to \bC$ is said to be entire real-analytic if $G(z)$ has an absolutely convergent Taylor series for every $z \in \bC$ and $G(0)=0$. This condition is satisfied in particular by polynomials in $z,\bar z$, as well as power-type nonlinearities such as $G(z)= \lambda |z|^{2k}z$ (with $\lambda \in \bC$ and $k \in \bN$) or $G(z)=\lambda z^k$ (with $\lambda \in \bC$ and $k \in \bN_+$).} function $G \colon \bC \to \bC$, with $G(0)=0$, consider the Cauchy problem 
	\begin{equation}\label{eq-pbm-pde}
		\begin{cases}
			\frac{1}{2\pi i} \partial_t u = \mu(D)u + V(t,x)G(u)\\
			u(0,x)=f(x)
		\end{cases} \qquad (t,x) \in \bR_+ \times \rd.
	\end{equation}
	Consider $R>0$. For every $f \in B_R \coloneqq \{f\in \cF L^1(\rd) : \|f\|_{\cF L^1} \le R\}$ there exist $T=T(R)\in (0,1]$ and a unique solution $u \in C^0([0,T];W^{1,\infty}(\rd))$, and the correspondence $f\mapsto u$ from $B_R$ to $C^0([0,T];W^{1,\infty}(\rd))$ is Lipschitz continuous. 
	
	Moreover, if $G(u)=u$ and $V \in L^\infty([0,+\infty);W^{\infty,1}(\rd))$, for every initial datum $f \in \cF L^1(\rd)$  there exists a unique global solution $u \in C^0([0,+\infty);W^{1,\infty}(\rd))$.
\end{theorem}
Primary examples of potentials to which Theorem \ref{thm-pde} applies are complex finite measures, which belong to $W^{\infty,1}(\bR^d)$ according to \cite[Proposition 3.4]{NT_cmp}. In fact, certain higher-order distributions are encompassed as well --- such as (distributional) derivatives of order $\leq (d-1)/2$ of the surface measure of any compact smooth hypersurface of $\bR^d$, with non-zero Gaussian curvature at every point. See Example \ref{ex-pot} below for further details in this connection. 

We emphasize that the solution $u$ of the problem \eqref{eq-pbm-pde} has the same spatial local regularity of the initial datum. To the best of our knowledge, this is the first time that point interactions can be approached via a phase space analysis angle. 

It is natural to wonder whether the assumption $m>2d$ can be relaxed, at least in the linear case. One may expect that such improvement can be possibly obtained at the price of a regularity gain at the level of potentials, which in turn could be balanced by the unlocking of less regular initial data. In fact, appropriate adjustments to the proof of Theorem \ref{thm-pde} allowed us to obtain the following result. 

\begin{theorem}\label{thm-pde-lowreg}
	Let $V \in  L^\infty([0,+\infty);W^{p,1}(\rd))$ for some $1 \le p \le \infty$, and let $\mu$ be as in Theorem \ref{thm-main} with $m\geq 2$ and $m > 2d/p'$, where $p'$ denotes the H\"older conjugate index associated with $p$. Consider the Cauchy problem	
	\begin{equation}
		\begin{cases}
			\frac{1}{2\pi i} \partial_t u = \mu(D) u + V(t,x)u\\
			u(0,x)=f(x),
		\end{cases} \qquad (t,x) \in \bR_+ \times \rd.
	\end{equation}
	For every initial datum $f \in \cF L^q(\rd)$ with $1 \le q \le p'$ there exists a unique solution $u \in C^0([0,+\infty);W^{q,\infty}(\rd))$. 
\end{theorem} 

We conclude by highlighting that some concrete choices of potentials that fit the assumptions of Theorem \ref{thm-pde-lowreg} are discussed in Example \ref{ex-pot} below. Remarkably, the latter include cropped Coulomb-type singularities and measures associated with certain low-dimensional smooth manifolds (or even fractal sets), the prime instance being the surface measure of the sphere $S^{d-1}\subset \rd$. 

\section{Preliminaries}
\subsection{Notation}
The inner product on $L^2(\rd)$ is defined by
\begin{equation}    \lan f,g \ran = \ird f(y) \overline{g(y)}dy, \qquad f,g \in L^2(\rd).
\end{equation}
This agreement extends to the duality pairing between a temperate distribution $f \in \cS'(\rd)$ and a function $g \in \cS(\rd)$ in the Schwartz class by setting $\lan f,g \ran = f(\overline{g})$.

The definition of the Fourier transform of $f \in L^2(\rd)$ used in this note is
\[  \wh{f}(\xi)=\cF(f)(\xi) = \ird e^{-2\pi i \xi\cdot x}f(x)dx, \]
with obvious extensions to temperate distributions. For $1\leq p\leq\infty$ we denote by $\cF L^p(\bR^d)$ the space of temperate distributions in $\bR^d$ whose (inverse) Fourier transform belongs to $L^p$, with the norm $\|f\|_{\cF L^p}=\|\wh{f}\|_{L^p}$. 

Given $A,B \in \bR$, we write $A \lesssim_\lambda B$ as a shorthand for the inequality $A \le C(\lambda) B$, where $C(\lambda)>0$ is a constant that depends on the parameter $\lambda$. We also write $A \asymp_\lambda B$ to mean that both $A \lesssim_\lambda B$ and $B \lesssim_\lambda A$ hold. 

The symbol $\lan x \ran$ stands for the inhomogeneous norm of $x \in \rd$, that is $\lan x \ran \coloneqq (1+|x|^2)^{1/2}$. 

\subsection{Basics of Gabor analysis} \label{sec-gabor} 
We recall here some elementary facts of time-frequency (alias phase-space) analysis. Proofs and further details can be found in the monographs \cite{benyi_book,CR_book,gro_book,NT_book}. 

Given $x,\xi \in \rd$, the translation and modulation operators $T_x$ and $M_\xi$, acting on $f \colon \rd \to \bC$, are defined by
\begin{equation}
    T_x f(y) \coloneqq f(y-x), \qquad M_\xi f(y) \coloneqq e^{2\pi i \xi \cdot y}f(y).
\end{equation}
The composition $\pi(x,\xi) \coloneqq M_\xi T_x$ is said to be the phase space shift of $f$ along $(x,\xi) \in \rd \times \wh{\rd} \simeq \rdd$.

The short-time Fourier transform (STFT), or Gabor transform, of $f \in \cS'(\rd)$ with respect to the window $g \in \cS(\rd)\smo$ is defined by
\begin{equation}
    V_g f(x,\xi) \coloneqq \lan f,\pi(x,\xi) g \ran = \cF (f \cdot \overline{T_x g})(\xi), \qquad (x,\xi) \in \rdd.
\end{equation} The Gabor transform is a continuous map $\cS(\rd)\times \cS(\rd) \ni f\times g \to V_gf \in \cS(\rdd)$, hence satisfying $|V_g f(x,\xi)| \lesssim_N \lan (x,\xi) \ran^{-N}$ for every $N \in \bN$. 

Given $1 \le p,q \le \infty$, the modulation space $M^{p,q}(\rd)$ is the space of all the distributions $f \in \cS'(\rd)$ such that, for a given $g \in \cS(\rd)\smo$, the mixed Lebesgue norm
\begin{equation}\label{eq:modsp-norm}
    \| f \|_{M^{p,q}_{(g)}} \coloneqq \| V_g f (x,\xi) \|_{L^q_\xi (L^p_x)} = \Big( \ird \Big( \ird |V_g f(x,\xi)|^p dx \Big)^{q/p} d\xi \Big)^{1/q}
\end{equation} is finite --- with obvious modifications if $p,q=\infty$. When $q=p$ we write $M^p(\rd)$ instead of $M^{p,p}(\rd)$. 

It is well known that $(M^{p,q}(\rd), \|\cdot\|_{M^{p,q}_{(g)}})$ is a Banach space for every $p,q$, and the norms $\|\cdot \|_{M^{p,q}_{(g)}}$ and $\|\cdot \|_{M^{p,q}_{(\gamma)}}$ are equivalent for every choice of a different window $\gamma \in \cS(\rd)\smo$ --- we thus omit the dependence on the window unless relevant, and write $\|\cdot\|_{M^{p,q}}$ from now on. 

Swapping the integration order in \eqref{eq:modsp-norm} yields
\begin{equation}\label{eq:wasp-norm}
    \| f \|_{W^{p,q}_{(g)}} \coloneqq \| V_g f (x,\xi) \|_{L^q_x (L^q_\xi)} = \Big( \ird \Big( \ird |V_g f(x,\xi)|^p d\xi \Big)^{q/p} dx \Big)^{1/q}. 
\end{equation}
The space of distributions $f \in \cS'(\rd)$ such that $\|f\|_{W^{p,q}_{(g)}}<\infty$ is again a Banach space, called the Wiener amalgam space $W^{p,q}(\rd)$, and different choices of windows in \eqref{eq:wasp-norm} result in equivalent norms on $W^{p,q}(\rd)$. We write $W^p (\rd)$ instead of $W^{p,p}(\rd)$. 

We list below some properties of modulation and Wiener amalgam spaces that are needed in the note. 

\begin{proposition}\label{prop-mod-was}
	Consider $1 \le p,p_1,p_2,q,q_1,q_2 \le \infty$. The following properties hold. 
	\begin{enumerate}
				\item We have the identification $W^{p,q}(\rd) = \cF M^{p,q}(\rd)$, that is 
		\begin{equation}
			f \in M^{p,q}(\rd) \quad \Longleftrightarrow \quad \wh{f} \in W^{p,q}(\rd).
		\end{equation} Moreover, $M^p(\rd) = W^p(\rd)$. 
		
		\item The pointwise multiplication is continuous $M^{p_1,q_1}(\rd) \times M^{p_2,q_2}(\rd) \to M^{p,q}(\rd)$ if and only if \begin{equation}
		\frac{1}{p_1}+\frac{1}{p_2}\ge \frac{1}{p}, \qquad \frac{1}{q_1}+\frac{1}{q_2}\ge 1+ \frac{1}{q}.
		\end{equation}
		In particular, $M^{\infty,1}(\rd)$ is a Banach algebra for pointwise multiplication and every modulation space is a Banach module over it. 
		
		\item The pointwise multiplication is continuous $W^{p_1,q_1}(\rd) \times W^{p_2,q_2}(\rd) \to W^{p,q}(\rd)$ if and only if \begin{equation}
	\frac{1}{p_1}+\frac{1}{p_2}\ge 1+\frac{1}{p}, \qquad \frac{1}{q_1}+\frac{1}{q_2}\ge \frac{1}{q}.
\end{equation}
In particular, $W^{1,\infty}(\rd)$ is a Banach algebra for pointwise multiplication and every Wiener amalgam space is a Banach module over it. 

		\item Given $\lambda \ne 0$ consider the dilation $f_\lambda(x)\coloneqq f(\lambda x)$. There exists a constant $C>0$ such that
		\begin{equation}
			\| f_\lambda \|_{M^{p,q}} \le C |\lambda|^{-d(1/p-1/q+1)}(1+\lambda^2)^{d/2} \|f\|_{M^{p,q}}. 
		\end{equation}
		
	\item Let $T$ be a linear operator and $A_1,A_2>0$ constants such that
	\begin{equation}
		\|T f \|_{M^{p_1,q_1}} \le A_1 \|f\|_{M^{p_1,q_1}}, \qquad \forall \, f \in M^{p_1,q_1}(\rd),
	\end{equation} 
	\begin{equation}
		\|T f \|_{M^{p_2,q_2}} \le A_2 \|f\|_{M^{p_2,q_2}}, \qquad \forall \, f \in M^{p_2,q_2}(\rd). 
	\end{equation}
	Then, for all $0<\theta<1$ and $1 \le p, q \le \infty$ such that 
	\begin{equation}
		\frac{1}{p} = \frac{1-\theta}{p_1}+\frac{\theta}{p_2}, \qquad \frac{1}{q} = \frac{1-\theta}{q_1}+\frac{\theta}{q_2},
	\end{equation} there exists a constant $C>0$ independent of $T$ such that 
	\begin{equation}
		\|T f \|_{M^{p,q}} \le CA_1^{1-\theta}A_2^\theta \|f\|_{M^{p,q}}, \qquad \forall \, f \in M^{p,q}(\rd). 
	\end{equation}
	
	\item Given $X \subseteq \cS'(\rd)$, set $X_{\text{comp}} \coloneqq \{ u \in X : \supp(u) \text{ is a compact subset of } \rd\}$. Then $(M^{p,q}(\rd))_{\text{comp}} = (\cF L^q(\rd))_{\text{comp}}$ and $(W^{p,q}(\rd))_{\text{comp}} = (\cF L^p(\rd))_{\text{comp}}$. 
	\end{enumerate}
	
\end{proposition}

\section{Proof of Theorem \ref{thm-main}}
Let us start by introducing the central objects of this note, namely generalized Fresnel functions of order $m$; cf.\ {\cite{BBR}}. 
	\begin{definition}
		Given $m \in \bR$, the class $A^m(\rd)$ consists of smooth functions $\mu \in C^\infty(\rd;\bR)$ such that, for all $\alpha \in \bN^d$, 
		\begin{equation}
			\sup_{x \in \rd} \lax^{|\alpha|-m}|\partial^\alpha \mu (x)| < \infty. 
		\end{equation}
		Given $\mu \in A^m(\rd)$, the corresponding Fresnel function is defined by $\Fmu(x)\coloneqq e^{2\pi i \mu(x)}$. 
	\end{definition}

The following result elucidates some phase space features of the Fresnel function $F_\mu$ in terms of the corresponding Gabor transform. 
	
	\begin{proposition}\label{prop-Fm}
		Let $\Fmu$ be the Fresnel function associated with $\mu \in A^m(\rd)$, $m\ge 2$. Let $g \in C^\infty_{\mathrm{c}}(\rd)$ be such that $\supp(g) \subseteq B=\{ y \in \rd : |y|\le 1\}$. 
		
		\begin{enumerate}[label=(\roman*)]\setlength\itemsep{0.5cm}
			\item We have 
			\begin{equation}\label{eq-VgFm-pointw}
				|V_g \Fmu (x,\xi)| \le \|g\|_{L^1}, \qquad (x,\xi) \in \rdd.
			\end{equation}
			\item There exists $A>0$ such that, for all $N\ge 0$, 
			\begin{equation}\label{eq-VgFm-decay}
				\text{ if } \quad |\xi-\nabla \mu(x)| \ge A |x|^{m-2} \quad \text{ then } \quad |V_g \Fmu (x,\xi)| \le C \lan \xi - \nabla \mu(x) \ran^{-N},
			\end{equation} for some $C>0$ that depends on $N$. 
		\end{enumerate}
	\end{proposition}
	
	\begin{proof}
		The pointwise bound in \eqref{eq-VgFm-pointw} is straightforward. 
		
		Concerning \eqref{eq-VgFm-decay}, let us separately discuss the cases where $|x|\le 1$ and $|x|>1$. It is easy to realize that if $|x|\le 1$ then the function 
		\begin{equation}
			h_x(y) \coloneqq e^{2\pi i \mu(y)}\overline{g(y-x)}
		\end{equation} belongs to $\cS(\rd)$, the corresponding seminorms being bounded with respect to $x$. Therefore, if $|x| \le 1$ then 
		\begin{equation}
			|V_g \Fmu (x,\xi)| = |\wh{ h_x}(\xi)| \lesssim_{N} \lan \xi \ran^{-N}, 
		\end{equation} for all $N \ge 0$, and this proves the claim since $\lan \xi - \nabla \mu(x) \ran \asymp \lan \xi \ran$ if $|x| \le 1$. 
		Assume now $|x|>1$, and note that 
		\begin{equation}
			V_g \Fmu (x,\xi) = \ird e^{-2\pi i \vp_{x,\xi}(y)} \overline{g(y)} dy, 
		\end{equation} where we set $\vp(y) = \vp_{x,\xi}(y) \coloneqq (y +x)\cdot \xi - \mu(y+x)$. After noticing that $\nabla \vp(y) = \nabla_y \vp(y) = \xi - \nabla \mu(y+x)$, it is natural to consider the set
		\begin{equation}
			\Omega \coloneqq \{ (x,\xi) \in \rdd : \quad |x|\ge 1, \quad |\xi-\nabla \mu(x)|\ge A |x|^{m-2}\},
		\end{equation} where the constant $A>0$ is chosen below. If $|y|\le 1$ and $|x| \ge 1$ then
		\begin{equation}
			|\nabla \mu(y+x)-\nabla \mu(x)| \le \sup_{|z| \le 1} |\nabla^2 \mu(z+x)| \le C |x|^{m-2},
		\end{equation} for some $C>0$. Choosing $A>2C$ yields, for $(x,\xi) \in \Omega$ and $|y|\le 1$,
		\begin{equation}\label{eq-vp-nabla}
			|\nabla \vp(y)| \ge |\xi - \nabla \mu(x)| -|\nabla \mu(y+x) - \nabla \mu(x)| \ge \frac{1}{2} |\xi- \nabla \mu(x)|\ge \frac{A}{2}|x|^{m-2}.
		\end{equation} Moreover, since $\mu \in A^m(\rd)$ we have 
		\begin{equation}\label{eq-vp-deriv}
			|\partial^\alpha \vp(y)| \lesssim |x|^{m-|\alpha|}, \qquad |\alpha|\ge 2, \quad |x|\ge 1, \quad |y|\le 1. 
		\end{equation} 

		Consider now the differential operator 
		\begin{equation}
			L \coloneqq \sum_{j=1}^d B_j^{(x,\xi)}(y) \, \partial_{y_j}, 
		\end{equation} where
		\begin{equation}
			B_j^{(x,\xi)}(y) \coloneqq \frac{\partial_{y_j} \varphi(y)}{-2\pi i |\nabla \varphi(y)|^2}, \qquad (x,\xi) \in \Omega, \quad |y|\le 1.
		\end{equation} 
    A recursive argument shows that $\partial^\alpha B_j^{(x,\xi)}$, $\alpha \in \bN^d$, is a (finite, complex) linear combination of terms such as
    \begin{align}
        \frac{\partial^{\gamma_1}\vp(y) \cdots \partial^{\gamma_k}\vp(y)}{|\nabla \vp(y)|^{k+1}} = \frac{\partial^{\gamma_1}\vp(y)}{|\nabla \vp(y)|} \frac{\partial^{\gamma_k}\vp(y)}{|\nabla \vp(y)|} \frac{1}{|\nabla \vp(y)|},
    \end{align} for suitable $k=k(j) \in \bN$, $\gamma_1,\ldots,\gamma_k \in \bN^d$ with $|\gamma_1|,\ldots,|\gamma_k| \ge 1$. Therefore, in light of \eqref{eq-vp-nabla} and \eqref{eq-vp-deriv} we infer that 
    \begin{equation}
        \frac{|\partial^{\gamma}\vp(y)|}{|\nabla \vp(y)|} \le \begin{cases}
            1 & (|\gamma|=1), \\ C_\gamma & (|\gamma| \ge 2), 
        \end{cases} \qquad (x,\xi) \in \Omega, \quad |y|\le 1,
    \end{equation} for suitable constants $C_\gamma >0$. Once again by virtue of \eqref{eq-vp-nabla}, we have
    \begin{equation}
        \frac{1}{|\nabla \vp(y)|} \le \frac{2}{|\xi-\nabla \mu(x)|}, \qquad (x,\xi) \in \Omega, \quad |y|\le 1,
    \end{equation} hence we conclude that 
    \begin{equation}\label{eq-derBj-bound}
        |\partial^\alpha B_j^{(x,\xi)}(y)| \lesssim_\alpha \frac{1}{|\xi-\nabla \mu(x)|}, \qquad (x,\xi) \in \Omega, \quad |y|\le 1, \quad \alpha \in \bN^d, \quad j=1,\ldots,d.
    \end{equation}

Note that 
		\begin{equation}
			L^N e^{-2\pi i \vp(y)} = e^{-2\pi i \vp(y)}, \quad N \in \bN. 
		\end{equation}
Integration by parts $N\ge 1$ times then gives
		\begin{equation}\label{eq-vgfm-bound}
			|V_g \Fmu(x,\xi)| \le \int_{|y|\le 1} |(L^\top)^N \bar g(y)| \dd{y}. 
		\end{equation} Arguing by induction on $N$ yields that $(L^\top)^N$ is a linear partial differential operator whose coefficients are (finite, complex) linear combinations of terms of the type
\begin{equation}
    \partial^{\gamma_1} B_{j_1}^{(x,\xi)}(y) \cdots \partial^{\gamma_N} B_{j_N}^{(x,\xi)}(y),
\end{equation} for suitable $j_1,\ldots,j_N \in \{1,\ldots,d\}$ and $\gamma_1,\ldots,\gamma_N \in \bN^d$. We have, as a result of \eqref{eq-derBj-bound}, 
\begin{equation}
   | \partial^{\gamma_1} B_{j_1}^{(x,\xi)}(y) \cdots \partial^{\gamma_N} B_{j_N}^{(x,\xi)}(y)|  \lesssim |\xi-\nabla \mu(x)|^{-N}, \qquad (x,\xi) \in \Omega, \quad |y|\le 1,
\end{equation} hence from \eqref{eq-vgfm-bound} we obtain
\begin{equation}
			|V_g \Fmu(x,\xi)| \lesssim |\xi-\nabla \mu(x)|^{-N}, 
		\end{equation} and the claim for $N\ge 1$ follows after noticing that if $(x,\xi) \in \Omega$ then $|\xi - \nabla \mu(x)|\ge A |x|^{m-2} \gtrsim 1$, hence $|\xi-\nabla \mu(x)| \asymp \lan \xi - \nabla \mu(x) \ran$. The case $N=0$ is trivially encompassed by \eqref{eq-VgFm-pointw}. 
	\end{proof}

In light of the previous findings, our next goal is to provide a characterization of $F_\mu$ in terms of membership to suitable modulation spaces. To this aim, we anticipate a technical result for later use. 
	\begin{proposition}\label{prop-cone}
		Let $\mu \colon \rd\smo \to \bR$ be a positively homogeneous $C^2$ function of degree $m \ge 2$, satisfying 
		\begin{equation}
			\det |\nabla^2 \mu(x)| \ne 0 \quad \text{ if } \quad |x|=1. 
		\end{equation}
		For all $x_0 \in \rd$ with $|x_0|=1$ there exist an open conic neighbourhood $\Gamma \subset \rd\smo$ of $x_0$ and a constant $C=C(x_0,\Gamma)>0$ such that 
		\begin{equation}\label{eq-nabla-cone}
			|\nabla \mu(x)-\nabla \mu(y)| \ge C |x-y|(|x|^{m-2}+|y|^{m-2}), \quad \forall\, x,y \in \Gamma. 
		\end{equation}
	\end{proposition}
	
	\begin{proof}
		As a consequence of the inverse mapping theorem we have that for suitable $c>0$, $r>1$ and $\veps>0$, if $r^{-1} < |x|,|y| < r$ and $x,y \in \Gamma_\veps \coloneqq \Big\{u \in \rd\smo : \Big| \frac{u}{|u|}-x_0 \Big| < \veps \Big\}$, then 
		\begin{equation}\label{ineq-1}
			|\nabla \mu(x)-\nabla \mu(y)| \ge c |x-y|. 
		\end{equation} This implies the inequality in \eqref{eq-nabla-cone} for $x,y\in\Gamma_\veps$, $r^{-1}<|x|,|y|<r$. We claim that, by homogeneity, this implies that the same estimate in \eqref{eq-nabla-cone} holds if $x,y \in \Gamma_\veps$ are such that $r^{-2}<|x|/|y|<r^2$. Indeed, given that both sides of \eqref{eq-nabla-cone} are positively homogeneous of degree $m-1$ with respect to the pair $x,y\not=0$, such inequality holds for a pair $x,y$ if and only if it holds for the pair $\alpha x, \alpha y$ for some $\alpha>0$. Since it has already been proved for $x,y\in \Gamma_\varepsilon$ with $r^{-1}<|x|,|y|<r$, it suffices to prove that, given $x,y\in \Gamma_\varepsilon$ with $r^{-2}<|x|/|y|<r^2$, there exists $\alpha>0$ such that $r^{-1}<\alpha|x|,\alpha|y|<r$, that is, $\alpha\in (r^{-1}/|x|,r/|x|)$ and $\alpha\in (r^{-1}/|y|,r/|y|)$. These two intervals have non-empty intersection --- if $|x|\leq|y|$ we have $r^{-1}/|y|\leq r^{-1}/|x|<r/|y|\leq r/|x|$, and analogous conclusions hold if $|x|\geq|y|$.
  
		It remains to prove the inequality in \eqref{eq-nabla-cone} for $x, y \in \Gamma_{\veps}$ satisfying $|x|/|y| \ge r^2$ or $|y|/|x| \ge r^2$ --- the arguments being identical, we focus here on the first assumption. Therefore, combining this with triangle inequality we infer 
		\begin{equation}
			|x-y|(|x|^{m-2}+|y|^{m-2}) \le \Big( 1+ \frac{1}{r^2} \Big) \Big( 1+\frac{1}{r^{2(m-2)}} \Big)|x|^{m-1}. 
		\end{equation}
		On the other hand, 
		\begin{align}
			|\nabla \mu(x) - \nabla \mu(y)| & = \Big| |x|^{m-1} \nabla \mu\Big(\frac{x}{|x|} \Big) - |y|^{m-1} \nabla \mu\Big(\frac{y}{|y|} \Big) \Big| \\
			& \ge \Big| |x|^{m-1} - |y|^{m-1} \Big| \Big|\nabla \mu\Big(\frac{x}{|x|}\Big)\Big| - |y|^{m-1} \Big|\nabla \mu\Big(\frac{x}{|x|}\Big) - \nabla \mu\Big(\frac{y}{|y|}\Big)\Big| \\
			& \ge |x|^{m-1}\Bigg( \Big( 1- \frac{1}{r^{2(m-1)}} \Big) \Big|\nabla \mu\Big(\frac{x}{|x|}\Big)\Big| - \frac{1}{r^{2(m-1)}} \sup_{\substack{u,v \in \Gamma_\veps \\ |u|=|v|=1}} |\nabla \mu(u)-\nabla \mu(v)|\Bigg).
		\end{align}
		Note that: 
		\begin{itemize} \item The term $\displaystyle \sup |\nabla \mu(u)-\nabla \mu(v)|$ can be made arbitrarily small by tuning the choice of $\veps$ small enough. 
			\item $|\nabla \mu(u)|$ is uniformly bounded away from zero if $|u|=1$. In fact, we have in general
			\begin{equation}\label{eq-grad-lowbd}
				|\nabla \mu(x)| \gtrsim |x|^{m-1}, \qquad |x| \ne 0, 
			\end{equation} as a consequence of Euler's homogeneous function theorem
			\begin{equation}
				\nabla^2 \mu(x) x = (m-1)\nabla \mu(x), \qquad x \in \rd,
			\end{equation} the assumptions on the invertibility of the Hessian of $\mu$ on the unit sphere and the fact that $m\ge 2$. 
		\end{itemize}
  Hence $|\nabla \mu(x)-\nabla \mu(y)| \gtrsim |x|^{m-1}$ for $x, y \in \Gamma_{\veps}$ satisfying $|x|/|y| \ge r^2$, which concludes the proof. 
	\end{proof}
	
	\begin{remark}
		The claim of Proposition \ref{prop-cone} fails to be true if $m<2$, as showed by straightforward computations in the case where $\mu(x)=|x|^{m}$ with $m<2$ and $d=1$.  Indeed, if $m<2$, the estimate 
  \[
  |mx^{m-1}-m|\geq C|x-1|(x^{m-2}+1),\qquad x>0
  \]
  clearly cannot hold  for some $C>0$, as one sees by letting $x\to+\infty$. Also, the localization in cones is necessary: if  $\mu(x)=x^3/3$, $x\in\mathbb{R}$, we have $|\mu'(x)-\mu'(y)|=|x^2-y^2|$, which vanishes (in particular) for $x=-y$, and therefore is not bounded below by $C|x-y|(|x|+|y|)$ for every $x,y\not=0$. 
  
	\end{remark}

We are now ready to locate the modulation space membership of higher-order Fresnel functions and prove Theorem \ref{thm-main}.
 
\begin{proof}[Proof of Theorem \ref{thm-main}]
		We show that the function $F_\mu(x)=e^{2\pi i\mu(x)}$ belongs to $ M^{1,\infty}(\rd)$, which is equivalent to the claim in view of Proposition \ref{prop-mod-was}.
		
		Let $\chi \colon \rd \to \bR$ be a smooth bump function near the origin, namely satisfying $0 \le \chi(x) \le 1$ for all $x \in \rd$, $\chi(x)=0$ if $|x|\ge 2$ and $\chi(x)=1$ if $|x|<1$. Therefore, we have
		\begin{equation}
			\Fmu(x)= \Fmu^{(1)}(x) + \Fmu^{(2)}(x), \quad \Fmu^{(1)}(x) \coloneqq \chi(x)e^{2\pi i \mu(x)}, \quad \Fmu^{(2)}(x) \coloneqq (1-\chi(x))e^{2\pi i \mu(x)}. 
		\end{equation}
		
		Let us prove that $\Fmu^{(1)} \in M^{1,\infty}(\rd)$ first. To this aim, it is enough to compute $V_g \Fmu^{(1)}$ with a smooth window function $g$ with compact support, for instance in the unit ball, so that $V_g \Fmu^{(1)}(x,\xi)= 0$ if $|x|\ge 3$. On the other hand, we obtain the pointwise bound $|V_g \Fmu^{(1)}(x,\xi)| \le \|g\|_{L^1}$, implying the claim. 
		
		We focus now on $\Fmu^{(2)}$. Let $\tilde\mu \colon \bR \to \bR$ be any smooth function such that $\tilde\mu(x)=\mu(x)$ if $|x|\ge 1$. The companion Fresnel function is $\wt \Fmu(x) = e^{2\pi i \tilde \mu(x)}$, and since $1-\chi \in M^{\infty,1}(\rd)$ is a bounded pointwise multiplier on any modulation space (cf.\ Proposition \ref{prop-mod-was}), it suffices to show that $\wt \Fmu \in M^{1,\infty}(\rd)$, namely that 
		\begin{equation}
			\sup_{\xi \in \rd} \ird |V_g \wt \Fmu (x,\xi)| \dd{x} < \infty. 
		\end{equation}
		
		It is clear that $\tilde\mu \in A^m(\rd)$, hence by Proposition \ref{prop-Fm} we have for every suitable window $g$ and $N>0$:
		\begin{equation} \label{eq-bound-wtfm}
			|V_g \wt \Fmu(x,\xi)| \lesssim_N \lan \xi - \nabla \tilde \mu(x) \ran^{-N} \quad \text{ if } \quad |\xi -\nabla \tilde \mu(x)|\ge A|x|^{m-2},
		\end{equation} for a suitable $A>0$. Given $\xi \in \rd$, consider the sets
		\begin{equation}
			\Omega_\xi^{(1)} \coloneqq \{ x \in \rd : |\xi - \nabla \tilde \mu(x)| > A|x|^{m-2}\}, 
		\end{equation}
		\begin{equation}
			\Omega_\xi^{(2)} \coloneqq \{ x \in \rd : |\xi - \nabla \tilde \mu(x)| \le A|x|^{m-2}\}.  
		\end{equation}
		
		We first prove that 
		\begin{equation}
			\sup_{\xi \in \rd} \int_{\Omega_\xi^{(2)}} |V_g \wt \Fmu (x,\xi)| \dd{x} < \infty. 
		\end{equation}
		In fact, this claim follows at once by combining \eqref{eq-VgFm-pointw} for $\wt\Fmu$ with uniformity of the measures of $\Omega_\xi^{(2)}$ over $\xi \in \rd$, that is \begin{equation}
			\sup_{\xi \in \rd} |\Omega_\xi^{(2)}| < \infty.
		\end{equation}
		The proof of the latter fact goes as follows. Let $\{\Gamma_j\}_{j=1, \ldots,n}$ be a finite covering of $\rd\smo$ consisting of open conic neighbourhoods obtained as in Proposition \ref{prop-cone}. It is then enough to show that 
		\begin{equation}
			\sup\{ \diam(\Omega_\xi^{(2)} \cap \Gamma_j) : \xi \in \rd, \, j=1,\ldots,n\} < \infty,
		\end{equation} which is in turn equivalent to
		\begin{equation}
			\sup\big\{ \diam( \{x \in \Omega_\xi^{(2)} \cap \Gamma_j, \, |x|\ge 1\}) : \xi \in \rd, \, j=1,\ldots,n\big\} < \infty.
		\end{equation}
		Note that if $x, \xi \in \Omega_\xi^{(2)}\cap \Gamma_j$ satisfy $|x|,|y|\ge 1$ then 
		\begin{equation}
			|\xi - \nabla \tilde \mu(x)| = |\xi - \nabla \mu(x)| \le A |x|^{m-2}, \qquad  |\xi - \nabla \tilde \mu(y)| = |\xi - \nabla \mu(y)| \le A |y|^{m-2}.
		\end{equation} As a result, we have
		\begin{equation}
			|\nabla \mu(x)-\nabla \mu(y)| \le A (|x|^{m-2}+|y|^{m-2}), 
		\end{equation} but by virtue of Proposition \ref{prop-cone} we also have, for a suitable $B>0$, 
		\begin{equation}
			|\nabla \mu(x)-\nabla \mu(y)| \ge B |x-y|(|x|^{m-2}+|y|^{m-2}), 
		\end{equation} hence resulting in the bound $|x-y| \le A/B$, which gives the claim. 
		
		It now remains only to prove that
		\begin{equation}
			\sup_{\xi \in \rd} \int_{\Omega_\xi^{(1)}} |V_g \wt \Fmu (x,\xi)| \dd{x} < \infty. 
		\end{equation} 
It follows from  \eqref{eq-bound-wtfm} and the definition of $\Omega_\xi^{(1)}$ that
		\begin{equation}\label{eq-bound-vgftm}
			|V_g \wt \Fmu (x,\xi)| \lesssim_N \lan x \ran^{-(m-2)N}, \qquad x \in \Omega_\xi^{(1)},
		\end{equation} which suffices to prove the claim in the case where $m>2$ provided that $N$ is chosen large enough. On the other hand, if $m=2$ then \eqref{eq-VgFm-pointw} yields uniform boundedness of $V_g \wt \Fmu$ on $\Omega_\xi^{(1)}$, which implies that 
		\begin{equation}
			\sup_{\xi \in \rd} \int_{E^{(\le 1)}_{\xi}} |V_g \wt \Fmu (x,\xi)| \dd{x} < \infty, \qquad E^{(\le 1)}_{\xi} \coloneqq \{ x \in \Omega_\xi^{(1)} : |x| \le 1 \}. 
		\end{equation}
		Therefore, it remains to obtain the analogous bound in the case where $E^{(\le 1)}_{\xi}$ is replaced by $E^{(>1)}_{\xi} \coloneqq \{x \in \Omega_\xi^{(1)} : |x|>1\}$. Let $\{\Gamma_j\}_{j=1, \ldots,n}$ be a finite covering of $\rd\smo$ consisting of open conic neighbourhoods obtained as in Proposition \ref{prop-cone}. By virtue of additivity, it is thus enough to show that 
		\begin{equation}
			\sup_{\xi \in \rd} \int_{E^{(>1)}_{\xi} \cap \Gamma_j} |V_g \wt \Fmu (x,\xi)| \dd{x} < \infty, \qquad \forall j=1, \ldots, n. 
		\end{equation}
		Since $\tilde \mu (x) = \mu(x)$ if $|x|>1$, we have $E^{(>1)}_\xi \cap \Gamma_j = \{ x \in \Gamma_j : |x|>1 \text{ and } |\xi - \nabla \mu(x)| \ge A \}$, which motivates the change of variables $v = \nabla \mu(x)$. The gradient map $\nabla \mu \colon \rd \smo \to \rd \smo$ is smooth, positively homogeneous of degree $1$ and by Proposition \ref{prop-cone} we have that $\Gamma_j$ is bijectively mapped to $\Lambda_j \coloneqq \nabla \mu (\Gamma_j)$. Therefore, since $| \det \nabla^2\mu(x)| \ge \delta >0$ for $|x|=1$ by assumption, after resorting to \eqref{eq-bound-wtfm} we obtain
		\begin{align}
			\int_{E^{(>1)}_{\xi} \cap \Gamma_j} |V_g \wt \Fmu (x,\xi)| \dd{x} & \lesssim_N \int_{E^{(>1)}_{\xi} \cap \Gamma_j} \lan \xi - \nabla \mu(x) \ran^{-N} \dd{x} \\
            & \lesssim \int_{\Gamma_j} \lan \xi - \nabla \mu(x) \ran^{-N} \dd{x}\\
			& \lesssim \int_{\Lambda_j} \lan \xi - v \ran^{-N} \dd{v} \\
            & \le \ird \lan v \ran^{-N} \dd{v} < \infty, 
		\end{align} as long as $N>d$. 
	\end{proof}
	
	\begin{remark} Let us make some comments on Theorem \ref{thm-main}. 
		\begin{itemize}\setlength\itemsep{1ex}
			\item Given a symmetric matrix $Q \in \bR^{d\times d}$, consider $\mu(x)=\frac{1}{2}Qx \cdot x$. After diagonalization, it is easy to realize that the assumption of invertibility of $\nabla^2 \mu(x)= Q$ is actually necessary in order for $F_\mu = e^{2\pi i \mu}$ to belong to $M^{1,\infty}(\rd)$. It is indeed well known that a linear and invertible change of variable yields a bounded operator in $M^{1,\infty}(\rd)$ (see for instance \cite[Proposition 2.3]{CR_book}). As a consequence, we can verify the claim assuming $\mu(x)=\frac{1}{2}\sum_{j=1}^k\veps_j x_j^2$, with $\veps_j\in \{-1,1\}$, $0\leq k\leq d$ (with $Q(x)\equiv 0$ if $k=0$). If $k<d$ then the modulus short-time Fourier transform $|V_g F_\mu(x_1,\ldots,x_d,\xi_1,\ldots,\xi_d)|$, with $g=g_1\otimes\ldots\otimes g_d$, is in fact independent of $x_d$, and therefore $\int_{\bR}|V_g F_\mu(x_1,\ldots,x_d,\xi_1,\ldots,\xi_d)|\,\dd{x_d}=\infty$ for every $x_1,\ldots,x_{d-1},\xi_1,\ldots, \xi_d$. As a consequence, $F_\mu\not\in M^{1,\infty}(\rd)$.
			
			\item It is well known \cite{benyi_unimod} that $e^{2\pi i |x|^m} \in W^{1,\infty}(\rd)$ if $m \le 2$, and this condition is in fact optimal. Indeed, if  $e^{2\pi i |x|^m} \in W^{1,\infty}(\rd)$ then by Proposition \ref{prop-mod-was} the pointwise multiplication by $e^{2\pi i |x|^m}$ is a bounded operator on $W^{p,q}(\bR^d)$ for every $1\leq p,q\leq\infty$ --- that is, the Fourier multiplier $e^{2\pi i |D|^m}$ is bounded on $M^{p,q}(\bR^d)$ for every $1\leq p,q\leq\infty$, which in turn happens only if $m\leq 2$ in light of \cite[Theorem 1.2] {miyachi}. 
   
            This result can now be complemented with Theorem \ref{thm-main}, showing that $e^{2\pi i |x|^m} \in M^{1,\infty}(\rd)$ if $m\ge 2$. In particular, at the threshold $m=2$ we have that the Fresnel function $e^{2\pi i |x|^2}$ belongs to $W^{1,\infty}(\rd) \cap M^{1,\infty}(\rd)$ --- this is of course known and can be easily verified by an explicit computation with a Gaussian window.
			
			\item Heuristic considerations lead us to believe that the assumption $m\ge 2$ is sharp. Indeed, the principle of duality of phases \cite[Sec.\ VIII.5]{stein} shows that, if $d=1$, the Fourier transform of $e^{2\pi i |x|^m}$ with $m>1$ roughly behaves like $e^{i c |\xi|^{m'}}$ for some $c \in \bR\smo$, where $m'$ satisfies $\frac{1}{m} + \frac{1}{m'} = 1$  --- see also \cite{hochberg} for an alternative derivation of this result. As argued above, the latter function notably fails to belong to $W^{1,\infty}(\bR) = \cF M^{1,\infty}(\bR)$ when $m'>2$. 
		\end{itemize}
	\end{remark}
	
\subsection{Some dispersive estimates for {$\mu(D)$}}
Let us highlight an interesting consequence of Theorem \ref{thm-main}, namely dispersive estimates for higher-order Schr\"odinger-type equations. 

\begin{corollary}\label{cor-disp-est}
	Let $\mu$ be as in Theorem \ref{thm-main}. Consider the Cauchy problem 
	\begin{equation}
		\begin{cases}
			\frac{1}{2\pi i} \partial_t u = \mu(D) u \\
			u(0,x)=f(x),
		\end{cases} \qquad (t,x) \in \bR_+ \times \rd.
	\end{equation}
	The corresponding propagator $U_\mu(t) \coloneqq e^{2\pi i t \mu(D)}$ is bounded $W^{\infty,1}(\rd) \to W^{1,\infty}(\rd)$, with 
	\begin{equation} \label{eq-Umut}
		\| U_\mu (t) f \|_{W^{1,\infty}} \le C \max\{t^{-d/m},t^{-2d/m}\} \|f\|_{W^{\infty,1}},
	\end{equation} for a constant $C>0$ that does not depend on $t$ and $f$. 
\end{corollary}
\begin{proof}
	Set $F_\mu(x)=e^{2\pi i\mu(x)}$. First, notice that $U_\mu(t)$ coincides with the Fourier multiplier $\Fmu(t^{1/m}D)$ with symbol $\Fmu(t^{1/m}x) = e^{2\pi i \mu(t^{1/m}x)}$. Consider now the case where $t=1$. Given $f \in W^{\infty,1}(\rd)$, in light of Proposition \ref{prop-mod-was} and Theorem \ref{thm-main} we have
	\begin{equation}
		\| U_\mu(1) f \|_{W^{1,\infty}} = \| \Fmu \wh f \|_{M^{1,\infty}} \lesssim \| \Fmu \|_{M^{1,\infty}} \| f\|_{W^{\infty,1}}. 
	\end{equation}
The general case reduces to the former by resorting to dilation properties of the symbol $\Fmu$ on modulation spaces (see Proposition \ref{prop-mod-was}): 
\begin{equation}
	\| \Fmu(t^{1/m} \cdot) \|_{M^{1,\infty}} \lesssim t^{-2d/m}(1+t^{2/m})^{d/2} \|\Fmu \|_{M^{1,\infty}} \lesssim \begin{cases}
		t^{-d/m} \|\Fmu \|_{M^{1,\infty}} & (t\ge 1) \\ 
		t^{-2d/m} \|\Fmu \|_{M^{1,\infty}} & (0<t<1), 
	\end{cases}
\end{equation} and the claim follows. 
\end{proof}

\begin{remark}
	We emphasize that the initial datum $f$ belongs to $W^{\infty,1}(\rd)$, encompassing the case of complex finite measures --- cf.\ \cite[Proposition 3.4]{NT_cmp}. Moreover, in the Schr\"odinger setting ($m=2$) we retrieve the asymptotics provided by well-known dispersive estimates on amalgam spaces, cf.\ \cite{CN_str}. In fact, arguing as in \cite{CN_str}, one can obtain Strichartz estimates in this general context. 
\end{remark}

\section{Proof of Theorem \ref{thm-pde}}
Let us now consider the problem of local wellposedness in presence of potential perturbations and nonlinearities. Our strategy is centered on a standard abstract iteration argument, which is stated below for the convenience of the reader. 

\begin{proposition}[{\cite[Proposition 1.38]{tao book}}]\label{prop-abs-nonlin}
	Let $X$ and $Y$ be two Banach spaces, and let $A \colon X\to Y$ be a bounded linear operator such that 
	\begin{equation}\label{eq-abs-cond1}
		\| Au \|_Y \le C_0 \|u\|_X, 
	\end{equation} 
	for all $u \in X$ and some $C_0 >0$. Let $G \colon Y\to X$ be a nonlinear operator such that $G(0)=0$ and 
	\begin{equation} \label{eq-abs-cond2}
		\| G(u)-G(v) \|_X \le \frac{1}{2C_0} \| u-v \|_Y, 
	\end{equation} 
	for all $u,v$ in the ball $B_{R}(Y) = \{u \in Y : \| u \|_Y \le R \}$ for some $R >0$. 
	
	Then, for every $u_0 \in B_{R/2}(Y)$ there exists a unique solution $u \in B_{R}(Y)$ to the equation 
	\[ u=u_0 + AG(u), \]
	and the correspondence $u_0 \mapsto u$ is Lipschitz with constant at most $2$, that is $\| u \|_Y \le 2 \| u_0 \|_Y$. 
\end{proposition}

We are now ready to prove our main wellposedness result for higher-order dispersive equations with potentials in $W^{\infty,1}(\rd)$ and analytic nonlinearities.  

\begin{proof}[Proof of Theorem \ref{thm-pde}]
	After setting $U_\mu(t) \coloneqq e^{2\pi i t \mu(D)}$, let us start by rewriting the problem in integral form via Duhamel's formula:
	\begin{equation}
		u(t,x) = U_\mu(t) f(x) + 2\pi i \int_0^t U_\mu (t-s) \big[V(s,\cdot) G(u(s,\cdot))\big](x) \dd{s}.
	\end{equation}
	Our goal is now to ensure that the assumptions of Proposition \ref{prop-abs-nonlin} are satisfied with $X=Y=C^0([0,T];W^{1,\infty}(\rd))$, with $0<T \le 1$ to be determined later, $u_0(t)= U_\mu(t)f$ and $A$ being the Duhamel operator $A v(t) \coloneqq 2\pi i \int_0^t U_\mu(t-s) V(s,\cdot) v(s,\cdot) \dd s$. 
	
	Let us first note that, in view of the embedding $\cF L^1(\rd)\hookrightarrow W^{1,\infty}(\rd)$, we have
	\begin{equation}
		\|u_0(t)\|_{W^{1,\infty}} = \|U_\mu(t) f \|_{W^{1,\infty}} \lesssim \|U_\mu(t)f\|_{\cF L^1} = \|f\|_{\cF L^1}. 
	\end{equation}
 Moreover, the map $[0,+\infty) \ni t \mapsto u_0(t) \in W^{1,\infty}(\rd)$ is continuous. Indeed, for $t,t_0 \in [0,+\infty)$, by dominated convergence we infer
 \begin{equation}
     \|U_\mu(t)f - U_\mu(t_0)f \|_{W^{1,\infty}} \lesssim \| (F_\mu(t^{1/m}\cdot)- F_\mu(t_0^{1/m}\cdot)) \wh f\|_{L^1} \to 0 \qquad \text{ as } \quad t \to t_0. 
 \end{equation} As a result, $u_0 \in X$ and $\|u_0\|_X \le C \|f\|_{\cF L^1}$ for a constant $C>0$ that does not depend on $T$. 
	
	Concerning \eqref{eq-abs-cond2}, the Banach algebra property of $X$ under pointwise multiplication (Proposition \ref{prop-mod-was}) implies (cf.\ for instance \cite{CN_fou}) that there exists $C_R>0$ such that
	\begin{equation}
		\|G(u)-G(v)\|_{X} \le C_R \|u-v\|_{X} \quad \text{ if } \quad \max\{\|u\|_{X}, \|v\|_{X}\}\le R.
	\end{equation}
	
	When it comes to prove \eqref{eq-abs-cond1}, by virtue of Proposition \ref{prop-mod-was} and \eqref{eq-Umut} in the case where $t \in [0,T]$, with $0<T\leq 1$, we have
	\begin{align}
		\|Au(t,\cdot)\|_{W^{1,\infty}} & \lesssim  \int_0^t \| U_\mu (t-s) V(s,\cdot)u(s,\cdot)\|_{W^{1,\infty}} \dd{s}  \\ 
		& \lesssim \int_0^t |t-s|^{-2d/m} \| V(s,\cdot)u(s,\cdot)\|_{W^{\infty,1}} \dd{s} \\
		& \lesssim \Big(\int_0^t |t-s|^{-2d/m} \dd{s} \Big) \|V\|_{L^\infty([0,1];W^{\infty,1}(\rd))} \|u\|_X \\
		& \le C' \frac{m}{m-2d} T^{\frac{m-2d}{m}} \|V\|_{L^\infty([0,1];W^{\infty,1}(\rd))} \|u\|_X,
	\end{align} for a constant $C'>0$ depending on $\mu$, where we used the assumption $m>2d$. It is now enough to choose \begin{equation}
		T = \min \Big\{\Big(\frac{m-2d}{2mC'C_R\|V\|_{L^\infty([0,1];W^{\infty,1}(\rd))}}\Big)^{\frac{m}{m-2d}}, 1\Big\}
	\end{equation} to obtain
	\begin{equation}
		\|Au(t,\cdot)\|_{W^{1,\infty}} \le \frac{1}{2C_R} \|u\|_X,\qquad 0\leq t\leq T,
	\end{equation} so that \eqref{eq-abs-cond1} is proved. 

In the case where $G(u)=u$ and $V\in L^\infty([0,+\infty);W^{\infty,1}(\rd))$, one can see that $C_R=1$ and the choice of $T$ does not depend on the initial datum $f$ anymore. Therefore, the mapping $A$ is a contraction and the procedure can be iterated on intervals of the form $[nT,(n+1)T]$, $n \in \bN_+$, leading to a global solution.
\end{proof}

\section{Proof of Theorem \ref{thm-pde-lowreg}}

\begin{proof}[Proof of Theorem \ref{thm-pde-lowreg}]
	Let $F_t^{(\mu)}$ be the symbol of the free propagator $U_\mu(t)=e^{2\pi i t \mu(D)}$, namely $F_t^{(\mu)}(\xi) = e^{2\pi i t \mu(\xi)}$. Note that Proposition \ref{prop-Fm} implies that $F_t^{(\mu)} \in M^\infty(\rd) $, uniformly with respect to $t$ --- that is, there exists $C >0$ such that  $\|F_t^{(\mu)}\|_{M^\infty} \le C$ for every $t \in \bR$. Combining this bound with \eqref{eq-Umut} via complex interpolation 
 (see Proposition \ref{prop-mod-was}), for all $1 \le p \le \infty$ we obtain
	\begin{equation}
		\|F_t^{(\mu)}\|_{M^{p,\infty}} \lesssim \max \{ t^{-\frac{d}{mp}}, t^{-\frac{2d}{mp}} \}.
	\end{equation} We point out that this estimate is used below with $p'$ in place of $p$. 
 
 We can run again the machinery of Proposition \ref{prop-abs-nonlin} as in the proof of Theorem \ref{thm-pde}, this time with $X=Y=C^0([0,T];W^{q,\infty}(\rd))$, $u_0(t) = U_\mu(t)f$ and $T>0$ to be determined later. In particular, we have now 
		\begin{equation}
		\|u_0(t)\|_{W^{q,\infty}} = \|U_\mu(t) f \|_{W^{q,\infty}} \lesssim \|U_\mu(t)f\|_{\cF L^q} = \|f\|_{\cF L^q},
	\end{equation} and the correspondence $[0,+\infty) \ni t \mapsto u_0(t) \in W^{q,\infty}(\rd)$ is continuous, so that $u_0 \in X$ with $\|u_0\|_X\leq C\|f\|_{\cF L^q}$ for a constant $C>0$ that does not depend on $T$.
 
	Using the convolution property $W^{r,1}(\rd) * W^{p',\infty}(\rd) \hookrightarrow W^{q,\infty}(\rd)$ with $1/p+1/q=1+1/r$ (see Proposition \ref{prop-mod-was}), for $0<t<1$ we infer
	\begin{align}
		\|U_\mu(t) f\|_{W^{q,\infty}} & = \| \big(\cF^{-1}F_t^{(\mu)}\big) * f\|_{W^{q,\infty}} \\
        & \lesssim \|F_t^{(\mu)} \|_{M^{p',\infty}} \|f \|_{W^{r,1}} \\
        & \lesssim t^{-\frac{2d}{mp'}} \|f\|_{W^{r,1}}. 
	\end{align}
	Moreover, the relationships $1/p+1/q=1+1/r$ and $1 \le q \le p'$ imply that $1 \le r \le \infty$, and by Proposition \ref{prop-mod-was} we infer that there exists $C>0$ such that, for every $t\geq0$ and $v\in W^{r,1}(\bR^d)$ 
	\begin{equation}
		\| V(t,\cdot) v \|_{W^{r,1}} \leq C \|V(t,\cdot)\|_{W^{p,1}} \|v\|_{W^{q,\infty}}.
	\end{equation}
	Combining these bounds, for $t \in [0,T]$ we obtain  
		\begin{align}
		\|Au(t,\cdot)\|_{W^{q,\infty}} & \lesssim  \int_0^t \| U_\mu (t-s) V(s,\cdot) u(s,\cdot)\|_{W^{q,\infty}} \dd{s}  \\ 
		& \lesssim \int_0^t |t-s|^{-\frac{2d}{mp'}} \| V(s,\cdot)u(s,\cdot)\|_{W^{r,1}} \dd{s} \\
		& \lesssim \Big(\int_0^t |t-s|^{-\frac{2d}{mp'}} \dd{s} \Big) \|V\|_{L^\infty([0,+\infty);W^{p,1})} \|u\|_X. 
	\end{align} The claim then follows as in the proof of Theorem \ref{thm-pde} after choosing $T$ conveniently small and iterating this argument. 
\end{proof}
	
\begin{remark}
	The condition $q \le p'$ is quite natural, since it does ensure that the product $Vu$ in the Cauchy problem is well-defined in the sense of distributions, since $V \in (\cF L^p)_{\text{loc}}(\rd)$ and $u \in (\cF L^{p'})_{\text{loc}}(\rd)$ imply $Vu \in (\cF L^1)_{\text{loc}}(\rd)$. 
\end{remark}

\begin{example} \label{ex-pot} It is worth mentioning some concrete examples of distributions that could serve as (time-independent) potentials in Theorem \ref{thm-pde-lowreg}.
	\begin{itemize} 
		\item Consider compactly supported Coulomb-type singularities as in $V(x)=|x|^{-\alpha} \textbf{1}_{B_R}(x)$, with $0<\alpha<d$ and $R>0$. In view of Proposition \ref{prop-mod-was}, $V$ belongs to $W^{p,1} (\rd)$ if and only if $V \in \cF L^p(\rd)$, namely if and only if $1/p' > \alpha/d$. In particular, the previous result holds provided that 
        \begin{equation}
            \frac{\alpha}{d} < \min \Big\{ \frac{1}{q}, \frac{m}{2d} \Big\}, \qquad m \ge 2,
        \end{equation}
  --- indeed, in that case there exists $1 \le p \le \infty$ such that $1/p' > \alpha/d$, $m >2d/p'$ and $q < p'$ (note that $\alpha/d<1$).
		
		\item Let $V$ be the surface measure of the sphere $S^{d-1} \subset \rd$. It is well known \cite{stein} that for $\xi \in \rd$ with $|\xi|$ sufficiently large the following asymptotic formula holds: 
        \begin{equation}
            \wh{V}(\xi) = 2|\xi|^{-\frac{d-1}{2}} \cos\Big(2\pi \Big(|\xi|-\frac{d-1}{8}\Big)\Big) + O(|\xi|^{-{\frac{d+1}{2}}}).
        \end{equation}
  Therefore $V \in \cF L^p (\rd)$ if and only if $1/p'>\frac{d+1}{2d}$. Once again in view of Proposition \ref{prop-mod-was}, we conclude that $V \in W^{p,1}(\rd)$ if and only if $1/p' > \frac{d+1}{2d}$, hence the previous result holds provided that 
        \begin{equation}
            \frac{d+1}{2d} < \min \Big\{ \frac{1}{q}, \frac{m}{2d} \Big\}
        \end{equation}
  (which implies $m > 2$)  --- indeed, in that case there exists $1 \le p \le \infty$ such that $1/p' > \frac{d+1}{2d}$, $m >2d/p'$ and $q < p'$. 
		
		The same conclusions can be inferred for surface measures of more general smooth compact hypersurfaces with non-zero Gaussian curvature at every point; indeed, if $V$ is such a measure we have 
  \begin{equation}\label{eq 19lu}
  |\widehat{V}(\xi)|\lesssim (1+|\xi|)^{-(d-1)/2},\qquad \xi\in\bR^d;
  \end{equation}
  see, e.g., \cite{stein}. Moreover, up to suitable adjustments, we can consider smooth low-dimensional manifolds of finite type or fractal sets in $\rd$ --- see \cite{wolff}.
		\item The case where $p=\infty$, $q=1$ in Theorem \ref{thm-pde-lowreg}, corresponds to well-behaved potentials and rough initial data --- indeed, the Feichtinger space $W^{1,1}(\rd) = M^1(\rd)$ consists of functions with local $\cF L^1$ regularity and $L^1$ decay at infinity, while $\cF L^\infty(\rd)$ is known as the space of pseudomeasures, and of course contains the space of finite complex measures. The solution $u \in C^0([0,+\infty);W^{\infty,\infty}(\rd))$ is therefore, for every $t\geq 0$,  locally a pseudomeasure. 
  \item Let $V$ be a (distributional) derivative of order $\leq (d-1)/2$ of the surface measure of a smooth compact hypersurface of $\bR^d$, with non-zero Gaussian curvature at every point. Then, in light of the discussion above, $\widehat{V}$ is a compactly supported distribution in $\cF L^\infty(\bR^d)$, and therefore belongs to $W^{\infty,1}(\bR^d)$ by Proposition \ref{prop-mod-was}. As a consequence, Theorem \ref{thm-pde} applies.  
	\end{itemize}
\end{example}

\section*{Acknowledgements}
The authors are members of Gruppo Nazionale per l’Analisi Matematica, la Probabilità e le loro Applicazioni (GNAMPA) --- Istituto Nazionale di Alta Matematica (INdAM). F.\ N.\ is a Fellow of the Accademia delle Scienze di Torino.  

The present research has been developed as part of the activities of the GNAMPA-INdAM  projects ``Analisi armonica e stocastica in problemi di quantizzazione e integrazione funzionale'', award number (CUP): E55F22000270001, and ``Tecniche analitiche e probabilistiche in informazione quantistica'', award number (CUP): E53C23001670001. We gratefully acknowledge financial support from GNAMPA-INdAM.

Part of this research has been carried out at CIRM-FBK (Centro Internazionale per la Ricerca Matematica --- Fondazione Bruno Kessler, Trento), in the context of the ``Research in Pairs 2023'' program. S.\ M.\ and S.\ I.\ T.\ are grateful to CIRM-FBK for the financial support and the excellent facilities. 

The authors report that there are no competing interests to declare. No dataset has been analysed or generated in connection with this note.

\end{document}